\documentclass[a4paper,11pt]{article}

\usepackage{amsmath}
\usepackage{amssymb}
\usepackage{amsthm}
\usepackage{graphicx}
\usepackage{amscd}
\usepackage{epic, eepic}
\usepackage{url}
\usepackage{color}
\usepackage[utf8]{inputenc} 
\usepackage{comment}
\usepackage{enumerate}   
\usepackage{todonotes}
\usepackage{graphicx}
\usepackage{epstopdf}
\usepackage{enumitem}
\usepackage{tikz}
\usepackage[top=22mm, bottom=23mm, left=22mm, right=22mm]{geometry}
\usepackage[bookmarks=false,hidelinks]{hyperref}

\makeatletter
\renewenvironment{proof}[1][\proofname] {\par\pushQED{\qed}\normalfont\topsep6\p@\@plus6\p@\relax\trivlist\item[\hskip\labelsep\bfseries#1\@addpunct{.}]\ignorespaces}{\popQED\endtrivlist\@endpefalse}
\makeatother

\newtheorem{theorem}{\bf Theorem}[section]
\newtheorem{lemma}[theorem]{\bf Lemma}
\newtheorem{corollary}[theorem]{\bf Corollary}
\newtheorem{proposition}[theorem]{\bf Proposition}

\theoremstyle{definition}

\newtheorem{definition}[theorem]{\bf Definition}

\def\eps{\varepsilon}

\title{Improved bounds for the Erd\H{o}s-Rogers $(s,s+2)$-problem}
\author{Oliver Janzer\thanks{Department of Pure Mathematics and Mathematical Statistics, University of Cambridge, United Kingdom. Research supported by a fellowship at Trinity College. Email: \textbf{oj224@cam.ac.uk}.}
	\and
	Benny Sudakov\thanks{Department of Mathematics, ETH Z\"urich, Switzerland. Research supported in part by SNSF grant 200021\_196965. Email: \textbf{benjamin.sudakov@math.ethz.ch}.}}
\date{}

\begin{document}

\maketitle

\begin{abstract}
    For $2\leq s<t$, the Erd\H os-Rogers function $f_{s,t}(n)$ measures how large a $K_s$-free induced subgraph there must be in a $K_t$-free graph on $n$ vertices. There has been an extensive amount of work towards estimating this function, but until very recently only the case $t=s+1$ was well understood. A recent breakthrough of Mattheus and Verstra\"ete on the Ramsey number $r(4,k)$ states that $f_{2,4}(n)\leq n^{1/3+o(1)}$, which matches the known lower bound up to the $o(1)$ term. In this paper we build on their approach and generalize this result by proving that $f_{s,s+2}(n)\leq n^{\frac{2s-3}{4s-5}+o(1)}$ holds for every $s\geq 2$. This comes close to the best known lower bound, improves a substantial body of work and is the best that any construction of similar kind can give.
\end{abstract}

\section{Introduction}

Given an integer $s\geq 2$, a set $U$ of vertices in a graph $G$ is said to be $s$-independent if $G[U]$ does not contain a copy of $K_s$. We write $\alpha_s(G)$ for the size of the largest $s$-independent set in $G$, so $\alpha_2(G)$ is just the usual independence number of $G$. Note that estimating Ramsey numbers is equivalent to estimating how small $\alpha_2(G)$ can be for a $K_t$-free graph on $n$ vertices. In 1962, Erd\H os and Rogers~\cite{ER62} initiated a natural generalization of this problem. For $2\leq s<t\leq n$, they defined $f_{s,t}(n)$ as the minimum of $\alpha_s(G)$ over all $K_t$-free graphs $G$ on $n$ vertices. This function is now commonly known as the Erd\H os-Rogers function, and the main research direction has been to estimate its growth as $n\rightarrow \infty$ for various fixed values of $s$ and $t$. (That said, we note that the problem of estimating, for fixed $s,m$ and $n\rightarrow \infty$, the smallest $t$ such that $f_{s,t}(n)\leq m$ has also been extensively studied: see \cite{BS23} and its references.) As remarked above, the case $s=2$ recovers the usual Ramsey problem: we have $f_{2,t}(n)<\ell$ if and only if $r(t,\ell)>n$.

For the case $s>2$, the first bounds were obtained by Erd\H os and Rogers \cite{ER62}, who showed that for every $s$ there exists a positive $\eps=\eps(s)$ such that $f_{s,s+1}(n)\leq n^{1-\eps}$. Noting that $f_{s,t}(n)\geq f_{s,t'}(n)$ holds for every $t<t'$, this implies the same bound for all pairs $(s,t)$. The first lower bound was given by Bollob\'as and Hind \cite{BH91}, who proved that $f_{s,t}(n)\geq n^{1/(t-s+1)}$. In particular, this showed that $f_{s,s+1}(n)\geq n^{1/2}$. Krivelevich \cite{Kri94,Kri95} improved these lower bounds by a small logarithmic factor and gave the general upper bound $f_{s,t}(n)=O(n^{\frac{s}{t+1}}(\log n)^{\frac{1}{s-1}})$. Later, the lower bound was significantly improved by Sudakov \cite{Sud05comb,Sud05rsa} for every $t\geq s+2$.

In the last decade or so, there has been significant progress on estimating $f_{s,s+1}$. First, Dudek and R\"odl \cite{DR11} proved that $f_{s,s+1}(n)=O(n^{2/3})$ for all $s$, bounding the exponent away from 1. Building on their approach but introducing further ideas, Wolfovitz \cite{Wol13} showed that $f_{3,4}(n)\leq n^{1/2+o(1)}$, matching the lower bound up to the $o(1)$ term in the exponent. Finally, Dudek, Retter and R\"odl~\cite{DRR14} generalized this for all $s$ by proving that $f_{s,s+1}(n)\leq n^{1/2+o(1)}$, which once again matches the lower bound. More precisely, their result states that $f_{s,s+1}(n)=O(n^{1/2}(\log n)^{4s^2})$. The current best lower bound is $f_{s,s+1}(n)= \Omega(n^{1/2}(\frac{\log n}{\log \log n})^{1/2})$, due to Dudek and Mubayi \cite{DM14}, whereas the current best upper bound is $f_{s,s+1}(n)=O(n^{1/2}(\log n)^{3-\frac{4}{s+2}})$, due to Mubayi and Verstra\"ete \cite{MV24}.

With the case $t=s+1$ settled (up to logarithmic factors), it is natural to study what happens when $t=s+2$. The lower bound of Sudakov for this case is $f_{s,s+2}(n)\geq n^{1/2-\frac{1}{6s-6}}(\log n)^{\Omega(1)}$. Dudek, Retter and R\"odl showed that for every $s\geq 4$, we have $f_{s,s+2}(n)=O(n^{1/2})$ and asked if there exists $s\geq 3$ such that $f_{s,s+2}(n)=o(n^{1/2})$, This was answered affirmatively in a strong form by Gowers and Janzer \cite{GJ20}, who proved that for each $s\geq 3$ we have $f_{s,s+2}(n)\leq n^{1/2-\frac{s-2}{8s^2-18s+8}}(\log n)^{O(1)}$. They also improved the best upper bound throughout the range $s+2\leq t\leq 2s-1$.

In a recent breakthrough, Mattheus and Verstra\"ete \cite{MV23} proved that the Ramsey number $r(4,k)$ satisfies $r(4,k)=\Omega(\frac{k^3}{\log^4 k})$, which matches the known upper bound up to a factor of order $\log^2 k$. Expressed in terms of the Erd\H os-Rogers function, their result is equivalent to the bound $f_{2,4}(n)=O(n^{1/3}(\log n)^{4/3})$. The main result in this paper is a generalization of their bound to $f_{s,s+2}$ for all values of $s$.

\begin{theorem} \label{thm:ERnewbound}
    For every $s\geq 2$, $f_{s,s+2}(n)=O\left(n^{\frac{2s-3}{4s-5}}(\log n)^3\right)$.
\end{theorem}

\noindent We remark that we did not try to optimize the logarithmic term.

It is not hard to see that this improves the bound of Gowers and Janzer on $f_{s,s+2}$ for every $s$. (Using the inequality $f_{4,7}(n)\leq f_{4,6}(n)$, it also improves the best known bound for $f_{4,7}$.) The smallest case where the result is new is $f_{3,5}(n)\leq n^{3/7+o(1)}$, improving the previous bound $f_{3,5}(n)\leq n^{6/13+o(1)}$ and coming close to the lower bound $f_{3,5}(n)\geq n^{5/12+o(1)}$ of Sudakov.

Furthermore, there is some evidence suggesting that the bound in Theorem \ref{thm:ERnewbound} is tight. First, it is tight for $s=2$. Moreover, it can be shown that no construction of the kind used in all recent works on the Erd\H os-Rogers function (including those providing the tight results for $f_{s,s+1}$) can beat this bound (this will be explained in more detail in the concluding remarks).

\section{The proof}

The proof of Theorem \ref{thm:ERnewbound} has the same rough structure as that of the bound $r(4,k)=\Omega(\frac{k^3}{\log^4 k})$ in \cite{MV23}, but requires several new ideas and twists of their method.
The following graph, studied in \cite{ONan72}, played a crucial role in \cite{MV23}.

\begin{proposition}[\cite{ONan72} or \cite{MV23}] \label{prop:algebraic graph}
    For every prime $q$, there is a bipartite graph $F$ with vertex sets $X$ and $Y$ such that the following hold.
    \begin{enumerate}
        \item $|X|=q^4-q^3+q^2$ and $|Y|=q^3+1$.
        \item $d_F(x)=q+1$ for every $x\in X$ and $d_F(y)=q^2$ for every $y\in Y$.
        \item $F$ is $C_4$-free. \label{prop:C4-free}
        \item $F$ does not contain the subdivision of $K_4$ as a subgraph with the part of size $4$ embedded to $X$. \label{prop:subdivisionfree}
    \end{enumerate}
\end{proposition}

Throughout this section, let $s\geq 2$ be a fixed integer. Let $q$ be a prime and let $F$ be the graph provided by Proposition \ref{prop:algebraic graph}. We now construct a $K_{s+2}$-free graph $H$ on vertex set $X$ randomly as follows. 
For each $y\in Y$, partition $N_F(y)$ uniformly randomly as $A_1(y)\cup A_2(y)\cup \dots \cup A_s(y)$ and place a complete $s$-partite graph in $H$ with parts $A_1(y), A_2(y),\dots , A_s(y)$. The following lemma, combined with properties \ref{prop:C4-free} and \ref{prop:subdivisionfree} of Proposition \ref{prop:algebraic graph}, shows that $H$ is $K_{s+2}$-free with probability 1.

\begin{lemma} \label{lem:clique partition}
    Assume that the edge set of a $K_{s+2}$ is partitioned into cliques $C_1,\dots,C_k$ of size at most $s$. Then there exist four vertices such that all six edges between them belong to different cliques $C_i$.
\end{lemma}

\begin{proof}
    Without loss of generality, we may assume that $C_1$ has at least three vertices, otherwise the statement is trivial. Since $C_1$ has at most $s$ vertices, there exist distinct vertices $u$ and $v$ which do not belong to $C_1$. Note that the clique containing $u$ and $v$ has at most one element of $C_1$, as the cliques partition the edge set. Since $C_1$ has size at least three, there exist distinct vertices $x$ and $y$ in $C_1$ such that no clique contains $u$, $v$ and at least one of $x$ and $y$. This means that no clique contains at least three elements from the set $\{x,y,u,v\}$, so these four vertices are suitable.
\end{proof}

To see that Lemma \ref{lem:clique partition} implies that $H$ is $K_{s+2}$-free, assume that $H$ does contain a copy of $K_{s+2}$ on vertex set $S$. Note that by property \ref{prop:C4-free} of Proposition \ref{prop:algebraic graph}, for any edge $uv$ in the complete graph $H[S]$, there is a unique $y\in Y$ such that $u,v\in N_F(y)$. Hence, we can partition the edge set of $H[S]$ into cliques, one with vertex set $N_F(y)\cap S$ for each $y\in Y$ such that $|N_F(y)\cap S|\geq 2$. Moreover, any such clique has size at most $s$, for otherwise it would have to contain (at least) two vertices from some $A_i(y)$, meaning that there could not be an edge between these two vertices. Hence, by Lemma \ref{lem:clique partition}, there are four vertices in $S$ such that for any two of them there is a different common neighbour in $Y$ in the graph $F$, contradicting property \ref{prop:subdivisionfree} of Proposition \ref{prop:algebraic graph}.

Our key lemma, proved in Section \ref{sec:Ksfree sets}, is as follows. Here and below we ignore floor and ceiling signs whenever they are not crucial.

\begin{lemma} \label{lem:few Ks-free}
    Let $q$ be a sufficiently large prime and let $t=q^{2-\frac{1}{s-1}}(\log q)^{3}$. Then with positive probability the number of sets $T\subset X$ of size $t$ for which $H[T]$ is $K_s$-free is at most $(q^{\frac{1}{s-1}})^t$.
\end{lemma}

It is easy to deduce Theorem \ref{thm:ERnewbound} from this.

\begin{proof}[Proof of Theorem \ref{thm:ERnewbound}]
    Take an outcome of $H$ which satisfies the conclusion of Lemma \ref{lem:few Ks-free}. Let $\tilde{X}$ be a random subset of $X$ obtained by keeping each vertex independently with probability $q^{-1/(s-1)}$, and let $G_0=H[\tilde{X}]$. Then for each set $T\subset X$, the probability that $T\subset \tilde{X}$ is $(q^{-1/(s-1)})^{|T|}$. Hence, for $t=q^{2-1/(s-1)}(\log q)^{3}$, Lemma \ref{lem:few Ks-free} implies that the expected number of $K_s$-free sets of size $t$ in $G_0$ is at most 1. Removing one vertex from each such set, we obtain a $K_{s+2}$-free graph $G$ in which every vertex set of size $t$ contains a $K_s$. The expected number of vertices in $G$ is at least $|X|q^{-1/(s-1)}-1\geq \frac{1}{2}q^{4-1/(s-1)}-1$, so there exists an outcome for $G$ with at least $\frac{1}{2}q^{4-1/(s-1)}-1$ vertices. So, for each sufficiently large prime $q$, there is a $K_{s+2}$-free graph with at least $\frac{1}{2}q^{(4s-5)/(s-1)}-1$ vertices in which every vertex set of size $q^{(2s-3)/(s-1)}(\log q)^{3}$ contains a $K_s$. Using Bertrand's postulate, this implies that $f_{s,s+2}(n)=O(n^{\frac{2s-3}{4s-5}}(\log n)^{3})$, completing the proof.
\end{proof}

\subsection{The number of $K_s$-free sets} \label{sec:Ksfree sets}

In this section we prove Lemma \ref{lem:few Ks-free}. This is the part of our proof which differs the most from the corresponding argument in \cite{MV23}. Indeed, they proved this lemma for the case $s=2$ by arguing that the graph $H$ is locally ``dense" (with high probability), and therefore by a known result on the number of independent sets in locally dense graphs, the number of independent sets of size $t$ in $H$ is sufficiently small. We do not have an analogue of this result for $s$-independent sets, so we take a slightly different approach, using a version of the celebrated hypergraph container method \cite{BMS15,ST15}. In our approach it is crucial to build the containers in several small steps, and we need to get good control of all the induced subgraphs of $H$ that arise while we run the process. The following lemma will help us achieve this control. It states that with high probability in every large enough vertex set in $H$, we not only have many copies of $K_s$, but we have many complete $s$-partite subgraphs which have similar-sized (and large) parts. (Here and below logarithms are to base $2$.)

\begin{lemma} \label{lem:thereisgoodscale}
    Assume that $q$ is sufficiently large. Then with positive probability, for every $U\subset X$ with $|U|\geq 500s^2q^2$ there exists some $\gamma\geq |U|/q^2$ such that the number of $y\in Y$ with $\gamma/(10s)\leq|A_i(y)\cap U|\leq \gamma$ for all $i\in [s]$ is at least $|U|q/(8(\log q)\gamma)$.
\end{lemma}

\begin{proof}
    Partition $Y$ as follows. Let $$Y_0=\{y\in Y: |N_F(y)\cap U|\leq e_F(U,Y)/(2|Y|)\}$$ and for each $1\leq i\leq 2\log q$, let $$Y_i=\{y\in Y: 2^{i-1}e_F(U,Y)/(2|Y|)< |N_F(y)\cap U|\leq 2^ie_F(U,Y)/(2|Y|)\}.$$
    To see that these sets indeed partition $Y$, note that for each $y\in Y$, we have $|N_F(y)\cap U|\leq d_F(y)=q^2$ and $e_F(U,Y)/(2|Y|)=|U|(q+1)/(2|Y|)\geq 1$. Observe that $e_F(U,Y_0)\leq e_F(U,Y)/2$, so there is some $1\leq i\leq 2\log q$ such that $e_F(U,Y_i)\geq e_F(U,Y)/(4\log q)\geq |U|q/(4\log q)$. Let $\gamma=2^i e_F(U,Y)/(2|Y|)$. Note that $\gamma\geq e_F(U,Y)/|Y|=|U|(q+1)/(q^3+1)\geq |U|/q^2$. Now $|Y_i|\geq e_F(U,Y_i)/\gamma\geq |U|q/(4(\log q)\gamma)$. Let $y\in Y_i$ and let $j\in [s]$. Note that $A_j(y)\cap U$ is a random subset of $N_F(y)\cap U$ which contains each $x\in N_F(y)\cap U$ independently with probability $1/s$. Hence, the expected value of $|A_j(y)\cap U|$ is $|N_F(y)\cap U|/s\geq \gamma/(2s)$. Therefore by the Chernoff bound (see, e.g., Theorem 4 in \cite{goemanschernoff}), the probability that we have $|A_j(y)\cap U|\leq \gamma/(10s)$ is at most $\exp(-\gamma/(16s))$. By the union bound, the probability that there exists some $j\in [s]$ such that $|A_j(y)\cap U|\leq \gamma/(10s)$ (for a fixed $y\in Y_i$) is at most $s\exp(-\gamma/(16s))\leq \exp(-\gamma/(32s))$ (where we used $\gamma\geq |U|/q^2\geq 500s^2$). These events are independent for different vertices $y\in Y_i$, so by the union bound the probability that this happens for more than $|Y_i|/2$ vertices is at most $\binom{|Y_i|}{|Y_i|/2}\exp(-|Y_i|\gamma/(64s))\leq \exp(-|Y_i|\gamma/(128s))\leq \exp(-|U|q/(512s\log q))$, where the last inequality follows from $|Y_i|\geq |U|q/(4(\log q)\gamma)$. We have therefore shown that for every $U\subset X$ of size at least $500s^2q^2$, the probability that a suitable $\gamma$ does not exist is at most $\exp(-|U|q/(512s\log q))$. The result follows after taking union bound over all choices for $U$ since $\sum_{u=1}^{|X|} \binom{|X|}{u}\exp(-uq/(512s\log q))\leq \sum_{u=1}^{\infty} q^{4u}\exp(-uq/(512s\log q))<1$, where the last inequality holds because $q^4\exp(-q/(512s\log q))<1/2$ when $q$ is sufficiently large.
\end{proof}

\begin{definition}
    Let us call an instance of $H$ \emph{nice} if it satisfies the conclusion of Lemma \ref{lem:thereisgoodscale}.
\end{definition}

Lemma \ref{lem:few Ks-free} can now be deduced from the following.

\begin{lemma} \label{lem:few Ks-free if nice}
    Let $q$ be sufficiently large and let $t=q^{2-1/(s-1)}(\log q)^{3}$. If $H$ is nice, then the number of sets $T\subset X$ of size $t$ for which $H[T]$ is $K_s$-free is at most $(q^{1/(s-1)})^t$.
\end{lemma}

In what follows, we will consider an $s$-uniform hypergraph on vertex set $X$ whose hyperedges correspond to the copies of $K_s$ in $H$. Then $K_s$-free subsets of $X$ will correspond to independent sets in this hypergraph, so to prove Lemma \ref{lem:few Ks-free if nice}, it suffices to bound the number of independent sets of certain size. This will be achieved using the hypergraph container method. For an $s$-uniform hypergraph $\mathcal{G}$ and some $\ell \in [s]$, we write $\Delta_{\ell}(\mathcal{G})$ for the maximum number of hyperedges in $\mathcal{G}$ containing the same set of $\ell$ vertices. Moreover, we write $\mathcal{I}(\mathcal{G})$ for the collection of independent sets in $\mathcal{G}$. 

The following result was proved in \cite{MSS18}. 

\begin{proposition}[{\cite[Theorem 1.5]{MSS18}}] \label{prop:BMScontainer}
    Suppose that positive integers $s$, $b$ and $r$ and a non-empty $s$-uniform hypergraph $\mathcal{G}$ satisfy that for every $\ell\in [s]$,
    $$\Delta_{\ell}(\mathcal{G})\leq \left(\frac{b}{v(\mathcal{G})}\right)^{\ell-1}\frac{e(\mathcal{G})}{r}.$$

    Then there exists a family $\mathcal{S}\subset \binom{V(\mathcal{G})}{\leq sb}$ and functions $f:\mathcal{S}\rightarrow \mathcal{P}(V(\mathcal{G}))$ and $g:\mathcal{I}(\mathcal{G})\rightarrow \mathcal{S}$ such that for every $I\in \mathcal{I}(\mathcal{G})$,
    $$I\subset f(g(I))$$
    and
    $$|f(g(I))|\leq v(\mathcal{G})-\delta r,$$
    where $\delta=2^{-s(s+1)}$.
\end{proposition}

\begin{corollary} \label{cor:BMScontainer}
    For every positive integer $s\geq 2$ and positive reals $p$ and $\lambda$, the following holds. Suppose that $\mathcal{G}$ is an $s$-uniform hypergraph with at least two vertices such that $pv(\mathcal{G})$ and $v(\mathcal{G})/\lambda$ are integers, and for every $\ell\in [s]$,
    $$\Delta_{\ell}(\mathcal{G})\leq \lambda\cdot p^{\ell-1}\frac{e(\mathcal{G})}{v(\mathcal{G})}.$$

    Then there exists a collection $\mathcal{C}$ of at most $v(\mathcal{G})^{spv(\mathcal{G})}$ sets of size at most $(1-\delta \lambda^{-1})v(\mathcal{G})$ such that for every $I\in \mathcal{I}(\mathcal{G})$, there exists some $R\in \mathcal{C}$ with $I\subset R$, where $\delta=2^{-s(s+1)}$.
\end{corollary}

\begin{proof}
    We apply Proposition \ref{prop:BMScontainer} with $b=pv(\mathcal{G})$ and $r=v(\mathcal{G})/\lambda$ and, after replacing $\mathcal{S}$ with $g(\mathcal{I}(\mathcal{G}))$, we take $\mathcal{C}=f(\mathcal{S})$. Then clearly $|\mathcal{C}|\leq |\mathcal{S}|\leq |\binom{V(\mathcal{G})}{\leq spv(\mathcal{G})}|\leq  v(\mathcal{G})^{spv(\mathcal{G})}$, where the last inequality holds since $\sum_{i=0}^k \binom{n}{i}\leq n^k$ for each $n\geq k\geq 2$.
    Moreover, any set in $\mathcal{C}$ has size at most $v(\mathcal{G})-\delta r=(1-\delta \lambda^{-1})v(\mathcal{G})$ and each $I\in \mathcal{I}(\mathcal{G})$ is contained in $f(g(I))\in \mathcal{C}$.
\end{proof}

Let $\mathcal{H}$ be the $s$-uniform hypergraph on vertex set $X$ in which a set of $s$ vertices form a hyperedge if they form a $K_s$ in $H$. The next lemma shows that if $H$ is nice, then a suitable subgraph of~$\mathcal{H}$ (chosen with the help of Lemma \ref{lem:thereisgoodscale}) satisfies the codegree conditions in Corollary \ref{cor:BMScontainer} with small values of $\lambda$ and $p$.

\begin{lemma} \label{lem:bounded degree}
    Assume that $H$ is nice. Then for each $U\subset X$ of size at least $500s^2q^2$ there exists a subgraph $\mathcal{G}$ of $\mathcal{H}[U]$ (on vertex set $U$) which satisfies
    \begin{equation}
        \Delta_{\ell}(\mathcal{G})\leq \lambda\cdot p^{\ell-1}\frac{e(\mathcal{G})}{v(\mathcal{G})} \label{eqn:bounded codegrees}
    \end{equation}
    for every $\ell\in [s]$ with $\lambda=O_s(\log q)$ and $p\leq |U|^{-1}q^{2-1/(s-1)}$.
\end{lemma}

\begin{proof}
    By Lemma \ref{lem:thereisgoodscale} there exists some $\gamma\geq |U|/q^2$ such that the number of $y\in Y$ with $\gamma/(10s)\leq|A_i(y)\cap U|\leq \gamma$ for all $i\in [s]$ is at least $|U|q/(8(\log q)\gamma)$. Let $p=(\gamma q^{\frac{1}{s-1}})^{-1}\leq |U|^{-1}q^{2-1/(s-1)}$. Let $E(\mathcal{G})$ consist of all $s$-sets $\{x_1,x_2,\dots,x_s\}$ in $U$ for which there exists $y\in Y$ with $\gamma/(10s)\leq|A_i(y)\cap U|\leq \gamma$ and $x_i\in A_i(y)\cap U$ for all $i\in [s]$. Clearly, such $x_1,x_2,\dots,x_s$ form a $K_s$ in $H$, so $\mathcal{G}$ is indeed a subgraph of $\mathcal{H}$.

    It remains to verify the codegree condition (\ref{eqn:bounded codegrees}).  Roughly speaking, the codegrees are small because for any set $S$ of at least two vertices in $U$, there is at most one vertex $y\in Y$ in the common neighbourhood of $S$ (since $F$ is $C_4$-free), and then all hyperedges in $\mathcal{G}$ containing $S$ live entirely in $N_F(y)$. More precisely, as we are only using those vertices $y\in Y$ to define hyperedges in $\mathcal{G}$ which satisfy $|A_i(y)\cap U|\leq \gamma$ for all $i$, we have $\Delta_{\ell}(\mathcal{G})\leq \gamma^{s-\ell}$ for each $2\leq \ell\leq s$. 
    Moreover, as $d_F(x)=q+1$ for all $x\in X$, we have $\Delta_1(\mathcal{G})\leq (q+1)\gamma^{s-1}$.

    On the other hand, note that $e(\mathcal{G})\geq \frac{|U|q}{8(\log q)\gamma}\cdot (\frac{\gamma}{10s})^s=\Omega_s(|U|q\gamma^{s-1}/\log q)$, so $e(\mathcal{G})/v(\mathcal{G})=\Omega_s(q\gamma^{s-1}/\log q)$. It follows that if $\lambda=C\log q$ for a sufficiently large constant $C=C(s)$, then $\lambda\cdot p^{\ell-1}\frac{e(\mathcal{G})}{v(\mathcal{G})}\geq 2q^{1-(\ell-1)/(s-1)}\gamma^{s-\ell}$. It follows that (\ref{eqn:bounded codegrees}) holds for each $1\leq \ell\leq s$.
\end{proof}

Combining Corollary \ref{cor:BMScontainer} and Lemma \ref{lem:bounded degree}, we prove the following result.

\begin{lemma} \label{lem:container}
    Let $q$ be sufficiently large and assume that $H$ is nice. Let $U$ be a subset of $X$ of size at least $500s^2q^2$. Now there exists a collection $\mathcal{C}$ of at most $(q^4)^{sq^{2-1/(s-1)}}$ sets of size at most $(1-\Omega_s((\log q)^{-1}))|U|$ such that for any $K_s$-free (in $H$) set $T\subset U$ there exists some $R\in \mathcal{C}$ with $T\subset R$.
\end{lemma}

\begin{proof}
    Choose a hypergraph $\mathcal{G}$ and parameters $\lambda,p$ according to Lemma \ref{lem:bounded degree}. By Corollary \ref{cor:BMScontainer}, there exists a collection $\mathcal{C}$ of at most $|U|^{sp|U|}$ sets of size at most $(1-2^{-s(s+1)}\lambda^{-1})|U|$ such that for every independent set $I$ in $\mathcal{G}$, there exists some $R\in \mathcal{C}$ such that $I\subset R$. The lemma follows by noting that any $K_s$-free set is an independent set in $\mathcal{G}$, $|U|\leq q^4$, $p\leq |U|^{-1}q^{2-1/(s-1)}$ and $\lambda=O_s(\log q)$.
\end{proof}

\begin{corollary} \label{cor:few sets}
    Let $q$ be sufficiently large and assume that $H$ is nice. Then there is a collection $\mathcal{C}$ of at most $(q^4)^{O_s(q^{2-1/(s-1)}(\log q)^2)}$ sets of size at most $500s^2q^2$ such that for any $K_s$-free (in $H$) set $T\subset X$ there exists some $R\in \mathcal{C}$ such that $T\subset R$.
\end{corollary}

\begin{proof}
    By Lemma \ref{lem:container}, there exists a positive constant $c_s$ such that whenever $U$ is a subset of $X$ of size at least $500s^2q^2$, then there is a collection $\mathcal{D}$ of at most $(q^4)^{sq^{2-1/(s-1)}}$ sets of size at most $(1-c_s(\log q)^{-1})|U|$ such that for any $K_s$-free set $T\subset U$ there exists some $R\in \mathcal{D}$ with $T\subset R$.

    We prove by induction that for each positive integer $j$ there is a collection $\mathcal{C}_j$ of at most $(q^4)^{jsq^{2-1/(s-1)}}$ sets of size at most $\max\left(500s^2q^2,(1-c_s(\log q)^{-1})^j|X|\right)$ such that for any $K_s$-free set $T\subset X$ there exists some $R\in \mathcal{C}_j$ with $T\subset R$. Note that, since $|X| \leq q^4$, by choosing $j$ to be a suitable integer of order $\Theta_s((\log q)^{2})$, the corollary follows. The base case $j=1$ is immediate by the first paragraph (applied in the special case $U=X$).

    Let now $\mathcal{C}_j$ be a suitable collection for $j$ and define $\mathcal{C}_{j+1}$ as follows. For each $U\in \mathcal{C}_j$ of size greater than $500s^2q^2$, take a collection $\mathcal{D}(U)$ of at most $(q^4)^{sq^{2-1/(s-1)}}$ sets of size at most $(1-c_s(\log q)^{-1})|U|$ such that for any $K_s$-free set $T\subset U$ there exists some $R\in \mathcal{D}(U)$ with $T\subset R$. Let $$\mathcal{C}_{j+1}=\{U\in \mathcal{C}_j:|U|\leq 500s^2q^2\}\cup \bigcup_{U\in \mathcal{C}_j:|U|>500s^2q^2} \mathcal{D}(U).$$
    Clearly, $|\mathcal{C}_{j+1}|\leq |\mathcal{C}_j|(q^4)^{sq^{2-1/(s-1)}}\leq (q^4)^{(j+1)sq^{2-1/(s-1)}}$.
    Moreover, since every set in $\mathcal{C}_j$ has size at most $\max\left(500s^2q^2,(1-c_s(\log q)^{-1})^j|X|\right)$, it follows that any set in $\mathcal{C}_{j+1}$ has size at most $\max\left(500s^2q^2,(1-c_s(\log q)^{-1})^{j+1}|X|\right)$. Finally, for any $K_s$-free set $T\subset X$ there exists some $U\in \mathcal{C}_j$ with $T\subset U$ and hence there exists some $R\in \mathcal{C}_{j+1}$ (either $U$ or some element of $\mathcal{D}(U)$) such that $T\subset R$. This completes the induction step and the proof.
\end{proof}

Corollary \ref{cor:few sets} implies that if $q$ is sufficiently large and $H$ is nice, then the number of $K_s$-free sets of size $t=q^{2-1/(s-1)}(\log q)^{3}$ in $H$ is at most $$(q^4)^{O_s(q^{2-1/(s-1)}(\log q)^{2})}\binom{500s^2q^2}{t}\leq (q^4)^{O_s(q^{2-1/(s-1)}(\log q)^{2})}(q^{1/(s-1)}/\log q)^t\leq (q^{1/(s-1)})^t,$$
proving Lemma \ref{lem:few Ks-free if nice}.

\section{Concluding remarks}

In this paper we constructed, for each fixed $s\geq 2$ and each $n$, a $K_{s+2}$-free $n$-vertex graph in which every vertex set of size at least roughly $n^{\frac{2s-3}{4s-5}}(\log n)^{3}$ induces a $K_s$.

It is not too hard to see that any construction which improves the exponent $\frac{2s-3}{4s-5}$ would have to have significantly fewer triangles than the random graph with the same edge density. Indeed, assume that $G$ is a $K_{s+2}$-free $n$-vertex graph with $e$ edges and $t$ triangles in which every subset of size at least $m$ induces a $K_s$, where $m\ll n^{\frac{2s-3}{4s-5}}$. Note that if $uv$ is an edge in $G$, then the common neighbourhood of $u$ and $v$ is $K_s$-free, so it must have size less than $m$. Hence, for each $i\geq 3$, the number of copies of $K_i$ in $G$ is at most $em^{i-2}$. On the other hand, every set of size $m$ contains a copy of $K_s$, so every set of size $2m$ contains at least $m$ copies, therefore, by a simple averaging argument, $G$ contains at least roughly $m\cdot (n/m)^s$ copies of $K_s$. Thus, comparing the upper and lower bound for the number of copies of $K_s$ in $G$, we have $em^{s-2}\geq n^s/m^{s-1}$, so $e\geq n^s/m^{2s-3}$. Since $m\ll n^{\frac{2s-3}{4s-5}}$, it follows that $e^2\geq n^{2s}/m^{4s-6}\gg n^3 m$ and hence $t\leq em\ll (e/n)^3$. This precisely means that $G$ has much fewer triangles than the random graph with the same edge density.

In particular, this shows that any construction improving our bound would have to be very different from the one used in this paper and from those used in \cite{DR11,Wol13,DRR14,GJ20}. Indeed, all these papers use so-called ``random block constructions" where we start with a suitable bipartite graph $F$ with parts $X$ and $Y$ and we define a new graph $G$ on vertex set $X$ by randomly placing complete $s$-partite graphs inside $N_F(y)$ for each $y\in Y$. Note that each $C_6$ in $F$ produces a triangle in $G$ with positive probability. Indeed, if we have a $C_6$ in $F$ with vertices $x_1,y_1,x_2,y_2,x_3,y_3$ in natural order, then $x_1x_2$ is an edge in $G$ with positive probability, since both $x_1$ and $x_2$ belong to $N_F(y_1)$. Similarly, $x_2x_3$ and $x_3x_1$ are edges with positive probability, thanks to their common neighbours $y_2$ and $y_3$, respectively. Since the complete $s$-partite graphs within $N_F(y)$ are placed independently of each other for all $y\in Y$, there is a positive probability that $x_1x_2x_3$ is a triangle in $G$. Now note that, unless $F$ is very sparse, it has at least as many $C_6$s as the random bipartite graph with the same edge density (i.e. $C_6$ supersaturates). It follows that $G$ has at least as many triangles as the random graph with the same edge density. Even after taking random induced subgraphs, this property remains, showing that any construction beating our bound would have to be significantly different.

\bibliographystyle{abbrv}
\bibliography{bibliography}

\end{document}